\documentclass[12pt]{article}

\usepackage[T1]{fontenc}
\usepackage{lmodern}
\usepackage{microtype}
\usepackage[margin=1in]{geometry}
\usepackage{amsmath,amssymb,amsthm}
\usepackage{thmtools}
\usepackage[numbers,sort&compress]{natbib}
\usepackage[hidelinks]{hyperref}
\hypersetup{
  pdftitle={The Zombie Damage Number of a Graph},
  pdfauthor={Randy Davila},
  pdfsubject={Graph pursuit and damage games},
  pdfkeywords={damage number, zombie number, pursuit-evasion, graph theory}
}

\newtheorem{theorem}{Theorem}
\newtheorem{proposition}[theorem]{Proposition}
\newtheorem{lemma}[theorem]{Lemma}
\newtheorem{corollary}[theorem]{Corollary}
\newtheorem{conjecture}[theorem]{Conjecture}
\newtheorem{problem}[theorem]{Problem}
\theoremstyle{definition}
\newtheorem{definition}[theorem]{Definition}
\theoremstyle{remark}

\newcommand{\dmg}{\operatorname{dmg}}
\newcommand{\zdmg}{\operatorname{zdmg}}

\begin{document}

\title{The Zombie Damage Number of a Graph}
\author{Randy Davila\thanks{Corresponding author:
\texttt{randy@firstprinciples.com}}\\
\small First Principles, 100 University Avenue, Toronto, Ontario M5J 1V6,
Canada\\
\small Department of Computational Applied Mathematics and Operations
Research,\\
\small Rice University, Houston, Texas 77005, USA}
\date{}

\maketitle

\begin{abstract}
In the damage variant of Cops and Robber, the \emph{damage number}
\(\dmg(G)\) is the number of distinct vertices damaged by the robber under
optimal play, with one cop minimizing and the robber maximizing this number.
We introduce the \emph{zombie damage number} \(\zdmg(G)\), obtained by
requiring the pursuer to move at every turn
along a shortest path toward the survivor. The parameter therefore measures
the cost of geodesic pursuit when the objective is containment rather than
capture alone. We prove that \(\dmg(G)\leq\zdmg(G)\) and characterize the
graphs with \(\zdmg(G)=0\). Geodesic pursuit incurs no additional damage on
trees, and we determine the parameter exactly for paths, cycles, split graphs,
and complete multipartite graphs. In particular,
\(\zdmg(C_n)=n\) for \(n\geq5\), while
\(\zdmg(K_{n_1,\ldots,n_k})=n_1+n_2-2\) when \(n_i\geq2\) for every \(i\).
We also develop a nonbacktracking trace argument for sparse graphs. If \(G\)
is connected, \(\delta(G)\geq2\), and \(g(G)\geq5\), then every vertex is
damaged, and hence \(\zdmg(G)=n(G)\). The same argument gives the sharp bound
\(\zdmg(G)\geq g(G)\) for every connected graph of finite girth at least five.
It follows that fully subdividing each edge of a connected graph of minimum degree at
least two produces a graph with maximum zombie damage. These results yield an
unbounded separation from the ordinary damage number. Most of the questions
leading to these results, as well as the open conjectures, arose through
\textsc{Theo-Conjecture}, an advisor-supervised discovery loop combining
automated conjecturing, exact computation, language-model-assisted
exploration, and human mathematical judgment.
\end{abstract}

\medskip
\noindent\textbf{Keywords:} damage number; zombie number; pursuit-evasion;
cograph; split graph; girth

\noindent\textbf{AMS subject classifications:} 05C57, 05C69, 91A43

\section{Introduction}\label{sec:introduction}

The game of Cops and Robber is a basic model of pursuit on a graph. In its
one-cop form, the cop first chooses an initial vertex of a graph \(G\), after
which the robber chooses an initial vertex with knowledge of the cop's
position. Play then proceeds in rounds. The cop moves first in each round and,
if capture does not occur, the robber moves next. On a move, either player may
pass or move to an adjacent vertex. Both players know the entire history and
the current positions. The cop wins by occupying the robber's vertex after a
cop move; otherwise, the robber may evade indefinitely. Nowakowski and
Winkler~\cite{NowakowskiWinkler1983} and, independently,
Quilliot~\cite{Quilliot1978} introduced the one-cop game, and Aigner and
Fromme~\cite{AignerFromme1984} initiated the study of the \emph{cop number}
\(c(G)\), the minimum number of cops that guarantee capture. Bonato and
Nowakowski~\cite{BonatoNowakowski2011} give a general account.

Cox and Sanaei retained the same order of play but changed the objective from
capture alone to containment~\cite{CoxSanaei2019}. If the robber occupies a
vertex \(u\), survives the next cop move, and then completes a robber move
from \(u\), the vertex \(u\) is \emph{damaged}. A pass is allowed and damages
the robber's current vertex, while capture on the preceding cop move prevents
that damage. Repeated visits to a previously damaged vertex do not increase
the score. The \emph{damage number} \(\dmg(G)\) is the value of this
perfect-information game with one cop: the cop minimizes, and the robber
maximizes, the total number of distinct damaged vertices. Thus \(\dmg(G)\)
can remain informative when one cop cannot force capture, because the cop may
still confine an escaping robber to a small region of the graph. Damage games
with several cops or robbers, throttling, and graph products have since been
studied in~\cite{CarlsonEtAl2021,CarlsonHalloranReinhart2022,HugganEtAl2024,StojakovicWulf2022}.

The deterministic game of Zombies and Survivors has the same initial
placement, information, turn order, and capture rule, but restricts every
pursuer move. Suppose that the zombie is at \(z\), the survivor is at
\(s\ne z\), and it is the zombie's turn. The zombie must move to a neighbor
\(z'\) satisfying
\[
d_G(z',s)=d_G(z,s)-1;
\]
equivalently, it must take the first edge of a shortest \(z\)-\(s\) path. The
zombie cannot pass, but it chooses among the legal moves when several shortest
paths are available. The survivor, like the robber, may move to a neighbor or
pass. Fitzpatrick et al.\ introduced this deterministic model
in~\cite{FitzpatrickEtAl2016}. Its \emph{zombie number} \(z(G)\) is the
minimum number of zombies that guarantee capture. Structural relations
between \(c(G)\) and \(z(G)\) on claw-free graphs are studied by Davila and
Henning~\cite{DavilaHenning2026}.

We combine the damage objective with this geodesic restriction. In the
\emph{zombie damage game}, one zombie and one survivor use the preceding
zombie rules, and a vertex is damaged precisely when it would be damaged in
the ordinary damage game. The \emph{zombie damage number} \(\zdmg(G)\) is the
resulting optimal number of distinct damaged vertices, with the zombie
minimizing and the survivor maximizing the score. Since every legal zombie
strategy is also available to an ordinary cop, but not conversely,
\[
\dmg(G)\leq\zdmg(G).
\]
The difference between the parameters measures the loss of containment caused
by the shortest-path restriction. This leads to two basic questions: when is
geodesic pursuit as effective as unrestricted pursuit, and which graph
structures allow a survivor to exploit the restriction?

Our results identify mechanisms on both sides. On every tree, the zombie can
be coupled to an optimal cop, and hence \(\zdmg(T)=\dmg(T)\). We determine
this common value on paths and compute the zombie damage number of every cycle.
For \(n\geq5\), a survivor can remain two steps ahead of the zombie while
traversing \(C_n\), so
\[
\zdmg(C_n)=n.
\]
Thus the difference \(\zdmg(G)-\dmg(G)\) is unbounded even on graphs of
maximum degree two.

Dense decompositions produce the opposite behavior. On a connected split
graph, the two damage parameters are zero if the graph has a universal vertex
and are one otherwise. For a complete multipartite graph
\(K_{n_1,\ldots,n_k}\), with
\(2\leq n_1\leq\cdots\leq n_k\), we prove
\[
\dmg(K_{n_1,\ldots,n_k})=1,
\qquad
\zdmg(K_{n_1,\ldots,n_k})=n_1+n_2-2.
\]
The same argument leads to a co-component upper bound and to an exact formula
conjecture for cographs.

Our principal sparse-graph argument is based on nonbacktracking walks. If the
girth is at least five, then every path of length two is the unique geodesic
between its endpoints. A survivor following a nonbacktracking walk therefore
forces the zombie to follow the same trace at distance two. This observation
gives the sharp bound
\[
\zdmg(G)\geq g(G)
\]
whenever the finite girth \(g(G)\) is at least five. If also
\(\delta(G)\geq2\), then a spanning nonbacktracking walk exists and
\(\zdmg(G)=n(G)\). Consequently, fully subdividing every edge of a connected
minimum-degree-two graph forces maximum zombie damage.

Exact finite-state calculations were used to formulate several further
questions. The resulting conjectures predict equality with the ordinary
damage number on chordal graphs, a sharp constant-factor bound on claw-free
graphs, and damage to every vertex of a bridgeless cubic graph \(G\) with
\(n(G)\geq10\). The girth theorem proves the latter prediction when triangles
and quadrilaterals are absent.

Computer assistance was used during the formulation of questions. The
computer system \textsc{Theo-Conjecture} compared exact values on finite
graphs and proposed relations for further study. A language model helped
interpret the candidates and suggest counterexample searches. The author
selected the problems, assessed their significance, and developed and checked
the proofs. This use of computation continues the graph-theoretic practice of
machine-assisted conjecturing initiated by Fajtlowicz's
\emph{Graffiti}~\cite{Fajtlowicz1988}. The present system developed from the
author's \emph{TxGraffiti}~\cite{DavilaTxGraffiti2026}; its
advisor-supervised discovery process is discussed
in~\cite{DavilaExecutableArtifacts2026}. A conjecture bearing the
\textsc{Theo-Conjecture} designation was first identified through this
process. The designation concerns its origin only: proofs and finite evidence
are stated separately.

The paper is organized as follows. Section~\ref{sec:prelim} gives the game
definitions and notation. Section~\ref{sec:basic-families} treats trees,
paths, and cycles. Section~\ref{sec:structure} develops the dense
decomposition arguments, including the exact complete multipartite formula.
Section~\ref{sec:sparse} introduces the nonbacktracking method for sparse
graphs. Section~\ref{sec:comparisons} records metric and pursuit comparisons.
Section~\ref{sec:conjectures} presents the remaining conjectures and their
mathematical motivation, and Section~\ref{sec:conclusion} concludes with open
problems.

\section{Preliminaries}\label{sec:prelim}

All graphs are finite, simple, and connected unless stated otherwise. The
order and size of a graph \(G\) are denoted by
\(n(G)=|V(G)|\) and \(m(G)=|E(G)|\), respectively.
For \(X\subseteq V(G)\), let \(G[X]\) denote the subgraph induced by \(X\),
and let \(\overline G\) denote the complement of \(G\). We write \(N_G(v)\)
for the open neighborhood of \(v\) and
\(N_G[v]=N_G(v)\cup\{v\}\) for its closed neighborhood; subscripts are omitted
when the graph is clear.
The minimum and maximum degrees are \(\delta(G)\) and \(\Delta(G)\). A
\emph{universal vertex} has degree \(n(G)-1\).

We denote the distance from \(u\) to \(v\) by \(d_G(u,v)\). The
\emph{radius} of \(G\) is
\[
\operatorname{rad}(G)=
\min_{v\in V(G)}\max_{u\in V(G)}d_G(u,v).
\]
The
\emph{girth} \(g(G)\) is the length of a shortest cycle and is infinite when
\(G\) is acyclic. A graph is \emph{cubic} if every vertex has degree three,
and it is \emph{bridgeless} if no edge is a bridge. A graph is
\emph{claw-free} if it has no induced \(K_{1,3}\), \emph{chordal} if it has no
induced cycle of length at least four, \emph{split} if its vertex set can be
partitioned into a clique and an independent set, and a \emph{cograph} if it
has no induced path on four vertices. A \emph{co-component} of \(G\) is the
subgraph induced by a component of \(\overline G\). Distinct co-components
are joined by all possible edges. The \emph{connected domination number}
\(\gamma_c(G)\) is the minimum cardinality of a connected dominating set. The
graph \(S(G)\) is obtained from \(G\) by subdividing every edge exactly once.

In Cops and Robber, the cop chooses an initial vertex, the robber then chooses
an initial vertex, and the cop moves first. At each turn either player may
move to a neighboring vertex or remain in place. Capture occurs if, after a
cop move, the cop occupies the robber's vertex. The minimum number of cops
that guarantee capture is the \emph{cop number} \(c(G)\). Deterministic
Zombies and Survivors has the same order of play, but every zombie must move
at each zombie turn to a neighbor closer to the survivor. The corresponding
minimum number of zombies that guarantee capture is the \emph{zombie number}
\(z(G)\).

In the one-cop damage game, a vertex is damaged when the robber occupies it
through a noncapturing cop turn and then completes a move from it, where
remaining in place is permitted. The \emph{damage number} \(\dmg(G)\) is the
number of distinct damaged vertices under optimal play, with the cop
minimizing and the robber maximizing this number~\cite{CoxSanaei2019}.

\begin{definition}
In the \emph{zombie damage game}, the pursuer must move at every turn. If the
zombie is at \(z\) and the survivor is at \(s\ne z\), then a legal zombie move
has the form \(z\to z'\), where \(z'\in N(z)\) and
\[
d(z',s)=d(z,s)-1.
\]
The survivor may move to a neighbor or pass. Damage is recorded by the same
rule as in the ordinary damage game. The \emph{zombie damage number}
\(\zdmg(G)\) is the resulting optimal payoff, with the zombie minimizing and
the survivor maximizing the number of distinct damaged vertices.
\end{definition}

When describing a strategy, we write \(u\to v\) for a move from \(u\) to
\(v\). A survivor may remain in place; we refer to this move as a \emph{pass}.
We say that \(G\) has \emph{full zombie damage} if
\(\zdmg(G)=n(G)\).

\begin{proposition}\label{prop:comparison}
If \(G\) is a connected graph, then
\[
\dmg(G)\leq \zdmg(G).
\]
\end{proposition}

\begin{proof}
Every legal zombie strategy is a legal cop strategy. The cop minimizes the
same payoff over a superset of the strategies available to the zombie.
\end{proof}

The zero case admits a complete characterization.

\begin{proposition}\label{prop:zero}
If \(G\) is a connected graph, then \(\zdmg(G)=0\) if and only if \(G\) has a
universal vertex.
\end{proposition}

\begin{proof}
If \(u\) is universal, the zombie starts at \(u\) and captures the survivor on
its first move, before any vertex is damaged. Conversely, suppose that the
zombie starts at a nonuniversal vertex \(z\). The survivor may start at a
nonneighbor of \(z\). The first zombie move cannot capture the survivor, which
may then pass and damage its initial vertex. Thus a strategy producing no
damage must start at a universal vertex.
\end{proof}

The same argument, with an unrestricted cop, recovers the corresponding zero
case for \(\dmg(G)\).

\section{Trees, paths, and cycles}\label{sec:basic-families}

On a tree, the unique geodesic from the pursuer to the evader orders pursuer
positions naturally. This makes it possible to compare a forced zombie with an
arbitrary cop move by move.

\begin{theorem}\label{thm:trees}
If \(T\) is a tree, then
\[
\zdmg(T)=\dmg(T).
\]
\end{theorem}

\begin{proof}
Fix an optimal one-cop damage strategy, and let the zombie start at the same
initial vertex as the cop. Against an arbitrary survivor strategy, run a
virtual copy of the cop strategy using the same survivor positions. Write
\(c\), \(z\), and \(s\) for the virtual cop, zombie, and survivor positions
immediately before a pursuer move.

We maintain the invariant that \(z\) lies on the unique \(c\)-\(s\) path,
with the vertices occurring in that order. It holds initially because
\(c=z\). Suppose that it holds before a pursuer move, and let \(c'\) and
\(z'\) be the new virtual-cop and zombie positions. The zombie moves one edge
from \(z\) toward \(s\). If the virtual cop moves toward \(s\), then it moves
at most one edge along the \(c\)-\(s\) path and cannot pass \(z'\). If it
passes or leaves this path, the unique \(c'\)-\(s\) path returns through
\(c\), and hence through \(z'\). Thus, unless the zombie captures the
survivor, \(z'\) lies on the \(c'\)-\(s\) path.

The survivor now makes the same move in the two games. If it passes or moves
away from the pursuers, the new path from \(c'\) to the survivor still
contains \(z'\). If it moves toward the pursuers, it cannot reach or cross
\(z'\), since such a move would be illegal in the zombie game. The invariant
is therefore preserved. The same observation shows that the survivor never
occupies \(c'\) without first occupying \(z'\), so every move legal against
the zombie is also legal in the virtual cop game.

Consequently, every vertex damaged against the zombie is damaged in the
coupled play against the virtual cop. If the zombie captures first, extend the
virtual play arbitrarily. The optimal cop strategy permits at most
\(\dmg(T)\) damaged vertices, so the zombie also permits at most
\(\dmg(T)\). Proposition~\ref{prop:comparison} gives the reverse inequality.
\end{proof}

The path formula refines Theorem~\ref{thm:trees} by determining the common
value.

\begin{theorem}\label{thm:paths}
If \(n\ge2\), then
\[
\zdmg(P_n)=\dmg(P_n)=\left\lfloor\frac n2\right\rfloor-1.
\]
\end{theorem}

\begin{proof}
The assertion is immediate for \(n=2,3\), so assume \(n\ge4\). Let the pursuer
start at a central vertex \(v\). Each component of \(P_n-v\) has at most
\(\lfloor n/2\rfloor\) vertices. If the evader starts in one such component,
the pursuer moves along the unique path toward it. After the first pursuer
move, the occupied vertex separates the evader from \(v\), and the evader is
confined to a terminal subpath containing at most
\(\lfloor n/2\rfloor-1\) vertices. This gives the upper bound for both an
ordinary cop and a zombie.

For the reverse inequality, consider an arbitrary initial pursuer vertex. One
component of its deletion contains a terminal path
\(x_1x_2\cdots x_\ell\), where \(x_1\) is adjacent to the pursuer and
\(\ell\ge\lfloor n/2\rfloor\). The evader starts at \(x_2\), moves successively
toward \(x_\ell\), and passes once after reaching \(x_\ell\). A pursuer cannot
enter this path ahead of the evader, and following from \(x_1\) preserves a
separation of one edge until the final pass. Hence the evader damages
\(x_2,\ldots,x_\ell\), a total of at least
\(\lfloor n/2\rfloor-1\) vertices. This strategy is valid against the more
powerful ordinary cop and therefore also against a zombie.
\end{proof}

Cycles behave differently because the survivor can convert the geodesic rule
into a prescribed direction of travel.

\begin{theorem}\label{thm:cycles}
For the cycle \(C_n\),
\[
\zdmg(C_n)=
\begin{cases}
0,&n=3,\\
2,&n=4,\\
n,&n\ge5.
\end{cases}
\]
\end{theorem}

\begin{proof}
The graph \(C_3\) has a universal vertex, so Proposition~\ref{prop:zero}
applies. On \(C_4\), label the vertices cyclically as
\(v_0,v_1,v_2,v_3\). After the zombie starts at \(v_0\), the survivor starts
at \(v_2\). If the zombie moves to \(v_1\), the survivor moves to \(v_3\) and
can damage both \(v_2\) and \(v_3\), regardless of the zombie's next choice.
Thus \(\zdmg(C_4)\ge2\).

For the reverse inequality on \(C_4\), the zombie again starts at \(v_0\).
An adjacent survivor is captured immediately, so suppose that it starts at
\(v_2\), and let the zombie move to \(v_1\). Only the move to \(v_3\) avoids
imminent capture after the first damaged vertex. The zombie then moves to
\(v_0\). Passing or moving to \(v_0\) permits immediate capture; moving back
to \(v_2\) returns, up to symmetry, to the same antipodal state and visits no
new vertex. Repeating this response confines the survivor to
\(\{v_2,v_3\}\), and hence \(\zdmg(C_4)\le2\).

Now let \(n\ge5\), and consider any initial zombie vertex. The survivor starts
at distance two from it. At this distance there is a unique geodesic first
edge from the zombie to the survivor. After the zombie takes that edge, the
survivor moves one edge in the same cyclic direction. The players are again
at distance two. Repeating the strategy prevents capture and moves the
survivor around the entire cycle. Every vertex is eventually damaged, so
\(\zdmg(C_n)\ge n\). The reverse inequality is immediate.
\end{proof}

Cox and Sanaei proved that
\(\dmg(C_n)=\lfloor(n-1)/2\rfloor\) for
\(n\ge4\)~\cite[Theorem~2.5]{CoxSanaei2019}. The preceding theorem therefore
gives a first quantitative answer to the cost of geodesic pursuit.

\begin{corollary}\label{cor:gap}
The difference \(\zdmg(G)-\dmg(G)\) is unbounded. More precisely, if
\(n\ge5\), then
\[
\zdmg(C_n)-\dmg(C_n)
=n-\left\lfloor\frac{n-1}{2}\right\rfloor.
\]
\end{corollary}

\section{Dense decompositions}\label{sec:structure}

The preceding examples show that sparse graphs may permit extensive damage.
In a dense graph, the complementary point of view is often more useful. At
the end of a survivor move, the survivor must occupy a nonneighbor of the
zombie; otherwise, the zombie captures it on the next move. Thus the relevant
question is not density by itself, but how the nonedges of the graph are
organized.

This section develops that principle through decompositions that localize
nonneighbors. We first show that the clique in a split partition confines the
survivor strongly enough to determine both damage parameters. We then pass to
co-components. Since vertices in distinct co-components are adjacent, every
nonneighbor of a vertex lies in its own co-component; a zombie alternating
between two suitably chosen vertices can therefore restrict all damage to two
explicit exceptional sets. This yields a general co-component upper bound.
For complete multipartite graphs the exceptional sets are entire parts, and
we prove that the bound is exact. The same viewpoint motivates the cograph
conjecture in Section~\ref{sec:conjectures}.

\subsection{Split graphs and co-components}

We begin with split graphs, where a clique controls every move out of the
independent part.

\begin{theorem}\label{thm:split}
If \(G\) is a connected split graph, then
\[
\dmg(G)=\zdmg(G)=
\begin{cases}
0,&\text{if \(G\) has a universal vertex},\\
1,&\text{otherwise}.
\end{cases}
\]
\end{theorem}

\begin{proof}
Let \(V(G)=C\cup I\), where \(C\) is a clique and \(I\) is an independent
set. The assertion in the universal-vertex case follows from
Proposition~\ref{prop:zero} and its ordinary-game analogue. Suppose that
\(G\) has no universal vertex, and let the zombie start at any vertex
\(c\in C\).

If the survivor starts adjacent to \(c\), then it is captured on the first
zombie move. Otherwise, it starts at a vertex \(u\in I\) that is not adjacent
to \(c\). Since \(G\) is connected, \(u\) has a neighbor \(x\in C\). The path
\(c,x,u\) is a geodesic, so the zombie may move from \(c\) to \(x\). The
survivor can now damage \(u\), but it cannot damage another vertex: passing
leaves it adjacent to \(x\), while every move from \(u\) enters the clique
\(C\) and again leaves it adjacent to \(x\). Hence \(\zdmg(G)\le1\). Since
\(G\) has no universal vertex, both damage parameters are positive; together
with Proposition~\ref{prop:comparison}, this proves the stated equality.
\end{proof}

The split partition is one instance of a dense decomposition that constrains
safe survivor moves. For a co-component \(H\), define
\[
q(H)=\min_{v\in V(H)}|V(H)-N_H[v]|
=n(H)-1-\Delta(H).
\]
Thus \(q(H)\) is the smallest number of nonneighbors that a vertex has within
its co-component.

\begin{proposition}\label{prop:cocomponent-upper}
Suppose that \(G\) is connected, \(\overline G\) is disconnected, and \(G\)
has no universal vertex. If \(H_1,\ldots,H_t\) are the co-components of \(G\)
and
\[
q(H_{(1)})\le q(H_{(2)})\le\cdots\le q(H_{(t)}),
\]
then
\[
\zdmg(G)\le q(H_{(1)})+q(H_{(2)}).
\]
\end{proposition}

\begin{proof}
For \(j\in\{1,2\}\), choose \(x_j\in V(H_{(j)})\) with exactly
\(q(H_{(j)})\) nonneighbors in \(H_{(j)}\). The zombie starts at \(x_1\)
and thereafter alternates along the edge \(x_1x_2\). A survivor that avoids
capture after the zombie reaches \(x_j\) must occupy a nonneighbor of \(x_j\).
Since every two vertices in distinct co-components are adjacent, every such
vertex lies in \(H_{(j)}\). Consequently, every damaged vertex belongs to
\[
\bigl(V(H_{(1)})-N_{H_{(1)}}[x_1]\bigr)
\cup
\bigl(V(H_{(2)})-N_{H_{(2)}}[x_2]\bigr),
\]
which has cardinality \(q(H_{(1)})+q(H_{(2)})\).
\end{proof}

The complete multipartite formula below shows that this upper bound can be
attained. Whether the recursive join structure of a cograph always forces
equality is recorded in Section~\ref{sec:conjectures}.

\subsection{Complete multipartite graphs}\label{sec:multipartite}

Complete multipartite graphs are the case in which every co-component is an
independent set. Although these graphs have many shortest paths, the game
reduces to the sequence of parts occupied by the two players. This gives an
exact formula and shows that the co-component bound can be sharp.

\begin{theorem}\label{thm:multipartite}
If \(G=K_{n_1,\ldots,n_k}\), where
\(k\ge2\) and \(n_1\le n_2\le\cdots\le n_k\), and \(n_1=1\), then
\[
\dmg(G)=\zdmg(G)=0.
\]
If \(n_1\ge2\), then
\[
\dmg(G)=1
\qquad\text{and}\qquad
\zdmg(G)=n_1+n_2-2.
\]
\end{theorem}

\begin{proof}
If \(n_1=1\), the vertex in the singleton part is universal, and
Proposition~\ref{prop:zero} applies. Suppose, therefore, that every part has
at least two vertices.

For the upper bound on the zombie damage number, choose
\(x\in V_1\) and \(y\in V_2\). The zombie starts at \(x\). A survivor starting
outside \(V_1\) is captured on the first zombie move, so assume that it starts
in \(V_1-\{x\}\). The zombie moves to \(y\) and thereafter alternates along
the edge \(xy\). After each zombie move, the survivor must enter the part now
occupied by the zombie; moving to any other part, or passing, leaves the
players adjacent and permits capture on the next zombie move. Hence every
noncapturing play is confined to
\((V_1-\{x\})\cup(V_2-\{y\})\), and
\(\zdmg(G)\le n_1+n_2-2\).

For the reverse inequality, consider any initial zombie vertex. The survivor
starts at a different vertex in the same part. Whenever the zombie moves into
a new part, the survivor follows it into that part and chooses a vertex other
than the zombie, preferring a vertex not previously damaged. This is always
possible because every part has cardinality at least two. The resulting infinite
sequence of occupied parts has no two consecutive terms equal, so at least two
parts occur infinitely often. In each such part the survivor eventually
damages all but at most the single vertex currently blocked by the zombie.
The two parts have total cardinality at least \(n_1+n_2\), and therefore the
survivor damages at least \(n_1+n_2-2\) vertices.

Finally, an unrestricted cop may start at any vertex and pass. The robber must
remain at a nonadjacent vertex in the same part to avoid capture, and hence can
damage only that vertex. Since \(G\) has no universal vertex, the ordinary
damage number is exactly one.
\end{proof}

\begin{corollary}\label{cor:complete-bipartite}
If \(a,b\ge2\), then
\[
\dmg(K_{a,b})=1
\qquad\text{and}\qquad
\zdmg(K_{a,b})=a+b-2.
\]
In particular, \(\zdmg(K_{a,b})-\dmg(K_{a,b})=a+b-3\).
\end{corollary}

The complete multipartite graph is claw-free precisely when no part has more
than two vertices. Thus Theorem~\ref{thm:multipartite} also supplies a natural
claw-free family. If \(\operatorname{CP}_k=K_{2,\ldots,2}\) is the cocktail
party graph on \(2k\) vertices, then
\[
\dmg(\operatorname{CP}_k)=1
\qquad\text{and}\qquad
\zdmg(\operatorname{CP}_k)=2
\]
for every \(k\ge2\).

\section{Nonbacktracking walks and sparse graphs}\label{sec:sparse}

We next isolate a local condition that allows the survivor to prescribe the
zombie's motion. A walk
\(v_0,v_1,\ldots\) is \emph{nonbacktracking} if
\(v_{i+1}\ne v_{i-1}\) whenever both terms are defined.

\begin{lemma}\label{lem:trace}
If \(g(G)\geq5\) and
\[
v_{-2},v_{-1},v_0,v_1,\ldots
\]
is a nonbacktracking walk, then the following holds. If the zombie starts at
\(v_{-2}\), the survivor
starts at \(v_0\), and thereafter the survivor follows the displayed walk,
then the zombie is forced to follow the same trace at distance two. Consequently, every
vertex occupied by the survivor is damaged.
\end{lemma}

\begin{proof}
For every \(i\ge-1\), the path
\(v_{i-1},v_i,v_{i+1}\) has length two and distinct endpoints. Its endpoints
are nonadjacent, since \(G\) is triangle-free, and they have no common
neighbor other than \(v_i\), since \(G\) has no quadrilateral. It is therefore
the unique geodesic between its endpoints. In particular, the zombie's first
move is forced from \(v_{-2}\) to \(v_{-1}\). After the survivor moves to
\(v_1\), the same argument applies successively to each three-vertex segment
of the walk. Thus the zombie occupies \(v_{i-1}\) when the survivor occupies
\(v_{i+1}\), and each survivor position is damaged before the survivor moves
to the next one.
\end{proof}

The trace lemma first yields a lower bound that does not require a minimum
degree hypothesis.

\begin{theorem}\label{thm:girth-lower}
If \(G\) is a connected graph with finite girth \(g(G)\ge5\), then
\[
\zdmg(G)\ge g(G),
\]
and this bound is sharp.
\end{theorem}

\begin{proof}
Let \(C\) be a cycle of length \(g(G)\), and fix the zombie's initial vertex
\(z\). Choose a shortest path from \(z\) to \(C\), and, upon first reaching
\(C\), continue around the cycle indefinitely in one direction. If \(z\) lies
on \(C\), simply begin by proceeding around \(C\). In either case this gives
an infinite nonbacktracking walk
\[
z=v_{-2},v_{-1},v_0,v_1,\ldots
\]
that visits every vertex of \(C\). The survivor starts at \(v_0\) and follows
the walk. Lemma~\ref{lem:trace} shows that all \(g(G)\) vertices of \(C\) are
damaged. For every integer \(r\geq5\), equality holds when \(G=C_r\) by
Theorem~\ref{thm:cycles}.
\end{proof}

\begin{lemma}\label{lem:spanning-nonbacktracking}
If \(G\) is a finite connected graph with minimum degree at least two, then
every oriented edge of \(G\) begins an infinite nonbacktracking walk that
visits every vertex of \(G\).
\end{lemma}

\begin{proof}
Form the directed nonbacktracking edge graph \(D(G)\). Its vertices are the
oriented edges, namely the ordered pairs \((u,v)\) for which \(uv\in E(G)\), and
\((u,v)(v,w)\) is an arc whenever \(w\ne u\). If \(G\) is a cycle, either
orientation of the cycle gives the required walk.

Suppose that \(G\) is not a cycle. Glover and Kempton proved that the
nonbacktracking matrix of a connected noncycle of minimum degree at least two
is irreducible~\cite[Proposition~2.3]{GloverKempton2021}. Equivalently,
\(D(G)\) is strongly connected. Starting from any prescribed oriented edge,
we may therefore concatenate directed paths in \(D(G)\) that reach an
oriented edge incident with each vertex of \(G\), and then append a directed
path back to the initial oriented edge. Repeating this closed directed walk
gives the required infinite nonbacktracking walk.
\end{proof}

\begin{theorem}\label{thm:girth-five}
If \(G\) is a connected graph with \(\delta(G)\ge2\) and \(g(G)\ge5\), then
\[
\zdmg(G)=n(G).
\]
Moreover, this conclusion is best possible with respect to the girth
hypothesis.
\end{theorem}

\begin{proof}
Fix the zombie's initial vertex \(v_{-2}\). Choose a neighbor \(v_{-1}\) of
\(v_{-2}\), and then choose
\(v_0\in N(v_{-1})-\{v_{-2}\}\), which is possible because
\(\delta(G)\ge2\). By
Lemma~\ref{lem:spanning-nonbacktracking}, extend the oriented edge
\((v_{-1},v_0)\) to an infinite nonbacktracking walk
\[
v_{-1},v_0,v_1,v_2,\ldots
\]
that visits every vertex. The survivor starts at \(v_0\) and follows this
walk. Lemma~\ref{lem:trace} shows that every vertex of \(G\) is damaged.

The girth threshold cannot be reduced: \(C_4\) has minimum degree two and
\(g(C_4)=4\), but Theorem~\ref{thm:cycles} gives
\(\zdmg(C_4)=2<n(C_4)\).
\end{proof}

\begin{corollary}\label{cor:cubic-girth}
If \(G\) is a cubic graph with \(g(G)\ge5\), then
\(\zdmg(G)=n(G)\).
\end{corollary}

Full subdivision eliminates the two local obstructions in
Theorem~\ref{thm:girth-five} and therefore gives a broad operation that forces
maximum damage.

\begin{corollary}\label{cor:subdivision}
If \(G\) is a connected graph with \(\delta(G)\ge2\), then
\[
\zdmg(S(G))=n(G)+m(G).
\]
\end{corollary}

\begin{proof}
The graph \(S(G)\) is connected, has minimum degree two, and satisfies
\(g(S(G))\ge6\). The result follows from Theorem~\ref{thm:girth-five}.
\end{proof}

\section{Metric and pursuit comparisons}\label{sec:comparisons}

The zombie damage number measures containment rather than capture. We first
record a metric lower bound inherited from the ordinary damage game. Huggan et
al.\ proved that \(\dmg(G)\ge \operatorname{rad}(G)-1\) for every connected
graph~\cite[Theorem~1.3]{HugganEtAl2024}.
Proposition~\ref{prop:comparison}
immediately transfers their bound to the zombie game.

\begin{corollary}\label{cor:radius}
If \(G\) is a connected graph, then
\[
\zdmg(G)\ge \operatorname{rad}(G)-1,
\]
and this bound is sharp.
\end{corollary}

\begin{proof}
The inequality follows from
\(\zdmg(G)\ge\dmg(G)\ge\operatorname{rad}(G)-1\). Equality holds for every
path by Theorem~\ref{thm:paths}, since
\(\operatorname{rad}(P_n)=\lfloor n/2\rfloor\).
\end{proof}

The complete bipartite formula gives a complementary obstruction to upper
bounds. It shows that the ordinary damage number, the two capture parameters,
connected domination, and minimum degree may all remain fixed while zombie
damage grows without bound.

\begin{proposition}\label{prop:k2b-comparison}
If \(b\ge2\), then
\[
\begin{aligned}
\zdmg(K_{2,b})&=b,
&
\dmg(K_{2,b})&=1,\\
c(K_{2,b})&=z(K_{2,b})=2,
&
\gamma_c(K_{2,b})&=\delta(K_{2,b})=2.
\end{aligned}
\]
Consequently, \(\zdmg(G)\) is not bounded above by any function of
\(\dmg(G)\), \(c(G)\), \(z(G)\), \(\gamma_c(G)\), and \(\delta(G)\).
\end{proposition}

\begin{proof}
Corollary~\ref{cor:complete-bipartite} gives
\(\dmg(K_{2,b})=1\) and \(\zdmg(K_{2,b})=b\). One pursuer cannot force
capture in either game: after each pursuer move, the evader can move to a
different vertex in the pursuer's new part, or pass if it is already there.
Two cops, or two zombies, may occupy the two-vertex part and capture an evader
in the other part on the first pursuer move. Hence
\(c(K_{2,b})=z(K_{2,b})=2\). A set containing one vertex from each part is a
connected dominating set, and no single vertex dominates the graph, so
\(\gamma_c(K_{2,b})=2\). Finally, \(\delta(K_{2,b})=2\). Since \(b\) is
arbitrary, the conclusion follows.
\end{proof}

Maximum degree behaves differently: on this family,
\(\zdmg(K_{2,b})=\Delta(K_{2,b})=b\). It cannot provide a universal lower
bound, however, since \(\zdmg(K_n)=0\) while \(\Delta(K_n)=n-1\).

\section{Open Conjectures}\label{sec:conjectures}

The preceding results isolate two mechanisms governing zombie damage. Dense
joins confine the survivor to the nonneighbors of a small number of vertices,
whereas sparse graphs may allow the survivor to prescribe a nonbacktracking
trace. The conjectures below ask how far these mechanisms extend in graph
classes with canonical decompositions or tightly controlled local structure.

We begin with cographs. Every connected cograph is a join of its
co-components, so Proposition~\ref{prop:cocomponent-upper} applies at the top
level of its recursive decomposition. Theorem~\ref{thm:multipartite} proves
that the bound is exact when the co-components are independent sets. This
suggests that, in a cograph, internal edges within a co-component do not give
the zombie a more efficient confinement strategy than alternating between two
co-components.

\begin{conjecture}[\textsc{Theo-Conjecture}]\label{conj:cograph}
If \(G\) is a connected cograph with no universal vertex, then
\[
\zdmg(G)=q(H_{(1)})+q(H_{(2)}),
\]
where \(H_{(1)}\) and \(H_{(2)}\) are defined as in
Proposition~\ref{prop:cocomponent-upper}.
\end{conjecture}

The cograph condition cannot simply be omitted: in a general join, an internal
common neighbor may let the zombie leave the alternating strategy and confine
the survivor more efficiently. Thus the conjecture is a statement about the
full recursive join structure, not merely about disconnected complements.

Trees and split graphs point to a second possible extension. Both are chordal,
and Theorems~\ref{thm:trees} and~\ref{thm:split} show that the geodesic
restriction costs no additional damage on either class. A chordal graph has a
perfect elimination ordering and decomposes along clique separators. These
features suggest that an optimal cop strategy can be transferred through the
clique decomposition without creating the cyclic ambiguity that a survivor
exploits on a long induced cycle.

\begin{conjecture}[\textsc{Theo-Conjecture}]\label{conj:chordal}
If \(G\) is chordal, then
\[
\dmg(G)=\zdmg(G)
\le \left\lfloor\frac{\gamma_c(G)}2\right\rfloor.
\]
Moreover, this bound is sharp.
\end{conjecture}

The coefficient in this conjecture would be best possible. Indeed,
\(\gamma_c(P_n)=n-2\) for \(n\geq3\), and Theorem~\ref{thm:paths} gives
\[
\zdmg(P_n)=\left\lfloor\frac{\gamma_c(P_n)}2\right\rfloor.
\]

Claw-free graphs behave differently. Cycles and cocktail-party graphs show
that equality of the ordinary and zombie damage numbers is already too much to
expect. Nevertheless, claw-freeness removes the complete-bipartite obstruction
from Proposition~\ref{prop:k2b-comparison}: the neighborhood of every vertex
has independence number at most two, so a survivor cannot encounter
arbitrarily many pairwise independent escape directions at one pursuer
position. This bounded local branching makes a universal multiplicative
comparison plausible.

\begin{conjecture}[\textsc{Theo-Conjecture}]\label{conj:clawfree-factor}
If \(G\) is claw-free, then
\[
\zdmg(G)\le 4\,\dmg(G).
\]
\end{conjecture}

The conjecture asks whether the local two-branch restriction can be converted
into a global four-to-one charging argument between vertices damaged in the
two games. More generally, it is the first case of a natural induced-star
problem.
For \(r\geq3\), let
\[
c_r=\sup\left\{
\frac{\zdmg(G)}{\dmg(G)}:
\begin{array}{l}
G\text{ is connected and }K_{1,r}\text{-free},\\
\dmg(G)>0
\end{array}
\right\},
\]
where \(K_{1,r}\)-free means that \(G\) contains no induced copy of
\(K_{1,r}\). The classes under consideration are nested, so the sequence
\((c_r)_{r\geq3}\) is nondecreasing. The complete bipartite formula gives the
rigorous lower bound
\[
c_r\geq
\frac{\zdmg(K_{r-1,r-1})}{\dmg(K_{r-1,r-1})}
=2r-4.
\]
In this notation, Conjecture~\ref{conj:clawfree-factor} asserts that
\(c_3\leq4\). The lower bound above shows that any general theory of the
constants \(c_r\) must grow at least linearly with \(r\).

\begin{problem}\label{prob:induced-star-factor}
Determine whether \(c_r\) is finite for every \(r\geq3\). If it is, determine
\(c_r\), or at least its asymptotic growth as \(r\) tends to infinity.
\end{problem}

Finally, consider cubic graphs. Cubic regularity alone is insufficient:
\(K_4\), \(K_{3,3}\), the triangular prism, and the three-dimensional cube do
not have full zombie damage. Bridges provide natural bottlenecks through which
the zombie may confine the survivor, but excluding bridges does not remove
2-edge cuts or other cyclic separations. Corollary~\ref{cor:cubic-girth}
proves full damage whenever the girth is at least five. Any further
obstruction must therefore use triangles or quadrilaterals to create competing
geodesic moves that cannot be avoided by a spanning survivor trace. In a
cubic graph, however, the attachments of a short cycle are tightly
constrained, and bridgelessness prevents any one attachment edge from
separating that cycle from the rest of the graph. This tension motivates the
following conjecture.

\begin{conjecture}[\textsc{Theo-Conjecture}]\label{conj:cubic}
If \(G\) is a bridgeless cubic graph with \(n(G)\geq10\), then
\[
\zdmg(G)=n(G).
\]
\end{conjecture}

Corollary~\ref{cor:cubic-girth} proves
Conjecture~\ref{conj:cubic} in the girth-at-least-five case. The unresolved
case asks whether the survivor can bypass every local ambiguity created by
triangles and quadrilaterals, even in the presence of cyclic edge cuts.

\section{Concluding remarks and further problems}\label{sec:conclusion}

The zombie damage number separates two roles of a pursuer that coincide in the
usual capture game. A zombie may fail to capture while still confining the
survivor, or its shortest-path restriction may allow the survivor to traverse
the entire graph. Trees and split graphs exhibit the first behavior. Cycles of
length at least five and, more generally, minimum-degree-two graphs of girth at
least five exhibit the second. Complete multipartite graphs lie between these
extremes: the survivor is confined to two parts, but the cardinalities of those parts
can make the resulting damage arbitrarily large.

The dense and sparse arguments also point to different structural programs.
For dense graphs, the relevant objects are modules, co-components, and clique
separators that restrict the survivor's safe positions. For sparse graphs, the
essential question is whether a survivor can prescribe a sufficiently large
nonbacktracking walk. The conjectures in Section~\ref{sec:conjectures} suggest
that these mechanisms extend well beyond the classes proved here.

\begin{problem}\label{prob:equality}
Characterize the connected graphs \(G\) for which
\(\zdmg(G)=\dmg(G)\). In particular, determine whether every chordal graph
satisfies this equality.
\end{problem}

\begin{problem}\label{prob:full-damage}
Characterize the connected graphs \(G\) for which
\(\zdmg(G)=n(G)\). In particular, determine whether every bridgeless cubic
graph \(G\) with \(n(G)\geq10\) has full zombie damage.
\end{problem}

\section*{Computational and methodological provenance}

Most of the problems, conjectures, and theorem directions developed in this
paper were stimulated by a \textsc{Theo-Conjecture} loop.
\textsc{Theo-Conjecture} is an advisor-supervised computer system for
mathematical discovery. In this study, exact ordinary and zombie damage values
and relevant structural annotations were assembled for a collection of
graphs. An automated conjecturing backend proposed candidate equalities,
inequalities, and statements for specified graph classes. OpenAI Codex assisted in
interpreting these candidates, designing counterexample searches, exploring
proof strategies, auditing arguments, organizing the manuscript, and revising
its language.

Candidate statements were checked against the available graphs and additional
independently generated examples. When a candidate failed, its counterexample
was retained and incorporated before the next conjecturing pass. The loop
therefore used failed statements to refine later questions rather than to
discard the surrounding line of inquiry. The author determined which
directions were mathematically significant, selected and formulated the
statements appearing here, developed the proofs, and independently checked
every mathematical claim. The computations and machine-generated suggestions
served as instruments of discovery; no theorem in this paper depends on
finite verification. The author accepts full responsibility for the article.

\section*{Data and code availability}

The finite-state solvers, graph-generation programs, and machine-readable data
supporting the computational statements are available from the author upon
reasonable request.

\bibliographystyle{abbrvnat}
\bibliography{references}

\end{document}